\documentclass{article}
\usepackage{amsmath,amssymb,amsfonts,amsthm,amscd}
\usepackage{epsf} 
\usepackage[dvipdfmx]{graphicx,color}
\usepackage{bm}
\usepackage{comment}

\newtheorem{theorem}{Theorem}[section]
\newtheorem{proposition}[theorem]{Proposition}
\newtheorem{lemma}[theorem]{Lemma}
\newtheorem{corollary}[theorem]{Corollary}

\theoremstyle{definition}

\theoremstyle{remark}

\newcommand{\vv}[1]{\ensuremath{\boldsymbol{#1}}}
\newcommand{\defeq}{\overset{\text{\tiny def}}{=}}

\begin{document}

\title{Multipliers and weak multipliers of algebras}

\author{Yuji Kobayashi and Sin-Ei Takahasi}

\date{{{ Laboratory of Mathematics and Games}\\
{\small (https:{\tiny //}math-game-labo.com)}\\
\vspace{4mm}
{\scriptsize 2020 MSC Numbers:
Primary 43A22; Secondary 17A99, 46J10. 
}
}}

\maketitle

\begin{abstract}
We study general properties of multipliers and weak multipliers of algebras.  We apply the results to determine the (weak) multipliers of associative algebras and zeropotent algebras of dimension 3 over an algebraically closed field.
\end{abstract}

\section{Introduction}

Multipliers of algebras, in particular, multipliers of Banach algebras, have been discussed in analysis. 
In this paper we will discuss them in a purely algebraic manner. 

Let $B$ be a Banach algebra. A mapping $T: B \rightarrow B$ is called a multiplier of $B$, if it satisfies the condition (I) $xT(y) = T(xy) = T(x)y\ (x, y \in B)$.
Let $M(B)$ denote the collection of all multipliers of $B$, and
let $B(B)$ be the collection of all bounded linear operators on $B$.
Then $M(B)$ forms an algebra and $B(B)$ forms a Banach algebra.
$B$ is called {\it without order} if it has no nonzero left or right annihilator.
If $B$ is without order, then 
$M(B)$ forms a commutative closed subalgebra of $B(B)$ (see \cite{Ka}, Proposition 1.4.11). In 1953, Wendel \cite{W} proved an important result that the multiplier algebra of $L^1(G)$ on a locally compact abelian group $G$ is isometrically isomorphic to the measure algebra on $G$.
The general theory of multipliers of Banach algebras has been developed by Johnson \cite{J}.
A good reference to the theory of multipliers of Banach algebra is given in Larsen \cite{L}.

When $B$ is without order, $T$ is a multiplier if it satisfies the condition
(II) $xT(y) = T(x)y\ (x, y \in B)$.  Many researchers had been unaware of difference between conditions (I) and (II) until Zivari-Kazempour \cite{Z} (see also \cite{Z2}) recently clearly stated the difference.
We call a mapping $T$ satisfying (II) a weak multiplier and denote the set of weak multipliers of $B$ by $M'(B)$.  Then, $M(B)$ is  in general a proper subset of $M'(B)$.
Furthermore, (weak) multipliers can be defined for an algebra $A$ not necessarily associative, and they are not linear mappings in general.
We denote the spaces of linear multipliers and linear weak multipliers 
of $A$ by $LM(A)$ and $LM'(A)$ respectively.
$M(A)$ and $LM(A)$ are subalgebras of the algebra
$A^A$ consisting of all mappings from $A$ to itself.  Meanwhile, $M'(A)$
and $LM'(A)$ are closed under the operation $\circ$ defined by
$T \circ S = TS + ST$, and they form a Jordan algebra.

In Sections 2 - 5 we study general properties of (weak) multipliers.
In particular, in sections 3 and 4 we give a decomposition theorem (Theorem 3.1),
and a matrix equation (Theorem 4.2) for (weak) multipliers.  They
play an essential role to analyze (weak) multipliers.


Complete classifications of associative algebras and zeropotent algebras of dimension 3 over an algebraically closed field of characteristic not equal to 2
were given in Kobayashi et al, \cite{KSTT} and \cite{KSTT2}.
In Sections 6 and 7 we completely determine the (linear) (weak) multipliers of 
those algebras.

\section{Multipliers and weak multipliers}
Let $K$ be a field and $A$ be a (not necessarily associative) algebra over $K$.
The set $A^A$ of all mappings from $A$ to $A$ forms an associative algebra
over $K$ in the usual manner.  Let $L(A)$ denotes the subalgebra of $A^A$ of all linear mappings from $A$ to $A$.

A mapping $T : A \rightarrow A$ is a \textit{weak multiplier} of $A$, if
\begin{equation} xT(y) = T(x)y \end{equation}
holds for any $x, y \in A$, and $T$ is a \textit{multiplier}, if 
\begin{equation} xT(y) = T(xy) = T(x)y \end{equation}
for any $x, y \in A$. Let $M(A)$ (resp. $M'(A)$) denote the set of all multipliers (resp. weak multipliers) of $A$.  Define
$$ LM(A) \defeq M(A) \cap L(A)\ \, \mbox{and}\ \, LM'(A) \defeq M'(A) \cap L(A).$$

\begin{proposition}
$M(A)$ (resp.  $LM(A)$) is a unital subalgebra of $A^A$ (resp.  $L(A)$), and
$M'(A)$ (resp.  $LM'(A)$) is a Jordan subalgebra of $A^A$ (resp.  $L(A)$).

\end{proposition}
\begin{proof}
First, the zero mapping is a multiplier because all the three terms in (2) are zero.
Secondly, the identity mapping is also a multiplier because the three terms in (2) are $xy$. 
Let $T, U \in M(A)$.  Then we have
\begin{equation}
x(T+U)(y) = xT(y) + xU(y) = T(xy) + U(xy) = T(x)y + U(x)y = (T+U)(x)y
\end{equation}
and
\begin{equation}
x(TU)(y) = xT(U(y)) = T(xU(y))  = TU(xy) = T( U(x)y) = (TU)(x)y
\end{equation}
for any $x, y \in A$.  Hence, $T+U, TU \in M(A)$.
Finally let $k \in K$, then 
\begin{equation} 
x(kT)(y) = kxT(y) = kT(xy) = kT(x)y = (kT)(x)y,
\end{equation}
and so $kT \in M(A)$.
Therefore, $M(A)$ is a unital subalgebra of $A^A$, and
$LM(A) = M(A)\cap L(A)$ is a unital subalgebra of $L(A)$.

Next, let $T, U \in M'(A)$.  Then, the equalities in (3)
and (5) hold
removing the center terms $T(xy)+U(xy)$ and $kT(xy)$ respectively.  
Hence, $M'(A)$ is a subspace of $A^A$.
Moreover, we have
$$ x(TU)(y) = xT(U(y)) = T(x)U(y) = U(T(x))y =UT(x)y$$
and similarly
$x(UT)(y) = TU(x)y$ for any $x, y \in A$.
Hence,
$$ x(TU + UT)(y) = (TU + UT)(x)y.$$
It follows that $TU + UT \in M'(A)$.\footnote{In general, for an associative algebra $A$ over a field $K$ of characteristic $\neq 2$, the
\textit{Jordan product} $\circ$ on $A$ is defined by $x\circ y = (xy + yx)/2$
for $x, y \in A$.}
\end{proof}
The opposite $A^{\rm op}$ of $A$ is the algebra with the same elements and the addition as $A$, but the multiplication $*$ in it is reversed, that is, $x*y = yx$ for all $x, y \in A$.
\begin{proposition}
$A$ and $A^{\rm op}$ have the same multipliers and weak multiplies, that is,
$$M(A^{\rm op}) = M(A)\ \ \mbox{and}\ \ M'(A^{\rm op}) = M'(A).$$
\end{proposition}
\begin{proof}
Let $T \in A^A$.  Then,
$T \in M'(A)$, if and only if 
$$x*T(y) = T(y)x = yT(x) = T(x)*y$$ 
for any $x, y \in A$, if and only if $T \in M'(A^{\rm op})$.  Further,
$T \in M(A)$, if and only if 
$$x*T(y) = T(y)x = T(yx) = T(x*y) = yT(x) = T(x)*y$$ 
for any $x, y \in A$, if and only if $T \in M(A^{\rm op})$.
\end{proof}

Let Ann$_l(A)$ (resp. Ann$_r(A)$) be the left (resp. right) annihilator of $A$
and let $A_0$ be their intersection, that is,
$$ {\rm Ann}_l(A) = \{a \in A\,|\, ax = 0 {\ \rm for\ all\ } x \in A\},$$
$$ {\rm Ann}_r(A) = \{a \in A\,|\, xa = 0 {\ \rm for\ all\ } x \in A\} $$
and
$$ A_0 = {\rm Ann}_l(A) \cap {\rm Ann}_r(A). $$ 
They are all subspaces of $A$, and when $A$ is an associative algebra, they are two-sided ideals.
For a subset $X$ of $A$, $\langle X\rangle$ denotes the subspace of $A$ generated by $X$.

\begin{proposition}  A weak multiplier $T$ of $A$ such that $\langle T(A)\rangle\cap A_0 = \{0\}$ is a linear mapping.
\end{proposition}
\begin{proof}
Let $x, y, z \in A$ and $a, b \in K$, and let $T$ be a weak multiplier.
We have
\begin{align*} T(ax + by)z & = (ax+by)T(z) = axT(z) + byT(z) = aT(x)z + bT(y)z\\
& = (aT(x) + bT(y))z. 
\end{align*}
Because $z$ is arbitrary, we have $w = T(ax+by) - aT(x) - bT(y) \in {\rm Ann}_l(A)$.  Similarly, we can show $w \in {\rm Ann}_r(A)$, and so $w \in A_0$. 
Hence, if $\langle T(A)\rangle \cap A_0 = \{0\}$,
then $w = 0$ because $w \in \langle T(A)\rangle$.  Since $a, b, x, y$
are arbitrary, $T$ is a linear mapping.
\end{proof}

\begin{corollary}
If $A_0 = \{0\}$, then any weak multiplier is a linear mapping over $K$, that is, $M'(A) = LM'(A)$ and $M(A) = LM(A)$.
\end{corollary}

\begin{proposition}
If $T$ is a weak multiplier, then $T({\rm Ann}_l(A)) \subseteq {\rm Ann}_l(A)$, $T({\rm Ann}_r(A)) \subseteq {\rm Ann}_r(A)$ and $T(A_0) \subseteq A_0$ .
\end{proposition}
\begin{proof}
Let $x \in {\rm Ann}_l(A)$, then for any $y \in A$ we have
$$ 0 = xT(y) = T(x)y. $$
Hence, $T(x) \in {\rm Ann}_l(A)$.  The other cases are similar.
\end{proof} 

In this paper we denote the subset $\{xy\,|\,x, y \in A\}$ of $A$ by $A^2$, though usually $A^2$ denotes the subspace of $A$ generated by this set.

\begin{proposition} 
Any mapping $T: A \rightarrow A$ such that $T(A) \subseteq A_0$ is a weak multiplier.  
Such a mapping $T$ is a multiplier if and only if $T(A^2) = \{0\}$.
In particular, if $A$ is the zero algebra, 
every mapping $T$ is a weak multiplier, and it is a multiplier if only if $T(0) = 0$.  
\end{proposition}
\begin{proof}
If $T(A) \subseteq A_0$, the both sides are 0 in (1) and $T$ is a weak multiplier.  This $T$ is a multiplier, if only if the term $T(xy)$ in the middle of (2) is $0$ for all $x, y \in A$, that is, $T(A^2) = \{0\}$. 
If $A$ is the zero algebra, then $A = A_0$ and $A^2$ = \{0\}.  Hence, any $T$ is a weak multiplier and it is a multiplier if and only if $T(0) = 0$.
\end{proof}

\section{Nihil decomposition}
Let $A_1$ be a subspace of $A$ such that 
\begin{equation}A = A_1 \oplus A_0.\end{equation}
Here, $A_1$ is not unique, but choosing an appropriate $A_1$ will become important later.
When $A_1$ is fixed, any mapping $T \in A^A$ is uniquely decomposed as
\begin{equation}
T = T_1 + T_0
\end{equation}
with $T_1(A) \subseteq A_1$ and $T_0(A) \subseteq A_0$.
We call (6) and (7) a {\it nihil decompositions} of $A$ and  $T$ respectively.
Let $\pi: A \rightarrow A_1$ be the projection and $\mu: A_1 \rightarrow A$ be the embedding, that is, 
$\pi(x_1+x_0) = \mu(x_1) = x_1$ for $x_1 \in A_1$ and $x_0 \in A_0$. 

Let $M_1(A)$ (resp. $M_0(A)$) denote the set of all multipliers $T$ of $A$ with $T(A) \subseteq A_1$ 
(resp. $T(A) \subseteq A_0$).  Similarly, the sets $M'_1(A)$ and $M'_0(A)$ of weak multipliers of $A$ are defined.  Also, set
$$LM_i(A) = M_i(A) \cap L(A)\ \,\mbox{and}\ \,LM'_i(A) = M'_i(A) \cap L(A)$$ 
for $i = 0,1$.  By Proposition 2.3 we see
$$M'_1(A) = LM'_1(A)\ \,\mbox{and}\ \,M_1(A) = LM_1(A),$$
and by Proposition 2.6 we have
\begin{equation} M'_0(A) = A_0^A,\, \, M_0(A) = \{T \in A_0^A\,|\,T(A^2) = \{0\}\}. \end{equation}


\begin{theorem}
Let $A$ = $A_1 \oplus A_0$ and $T = T_1 + T_0$ be nihil decompositions of $A$ and $T \in A^A$ respectively.

(i)  $T$ is a weak multiplier, if and only if $T_1$ is a weak multiplier.
If $T$ is a weak multiplier, $T_1$ is a linear mapping satisfying
$T_1(A_0) = \{0\}$. 

(ii)  If $T_1$ is a multiplier and $T_0(A^2) =  \{0\}$, then $T$ is a multiplier.
If $A_1$ is a subalgebra of $A$, the converse is also true.

Suppose that $A_1$ is a subalgebra of $A$, and let $\Phi$ be a mapping sending $R \in (A_1)^{A_1}$ to $\mu \circ R \circ \pi \in A^A$.  Then,

(iii)  $\Phi$ gives an algebra isomorphism from $M(A_1)$ onto $M_1(A)$ and a Jordan isomorphism from $M'(A_1)$ onto $M'_1(A)$. 
\end{theorem}
\begin{proof}
Let $x, y \in A$.  

(i)  If $T$ is a weak multiplier, then
$$ xT_1(y) = x(T(y) - T_0(y)) = xT(y) = T(x)y = T_1(x)y.$$
Thus, $T_1$ is also a weak multiplier.
Moreover, $T_1$ is a linear mapping by Proposition 2.3 and $T_1(A_0) \subseteq A_1 \cap A_0 = \{0\}$ 
by Proposition 2.5.
Conversely, if $T_1$ is a weak multiplier, then
$$ xT(y) = xT_1(y) = T_1(x)y = T(x)y, $$
and so $T$ is a weak multiplier.  

(ii)  If $T_1$ is a multiplier and $T_0(A^2) = 0$, then $T$ is a multiplier because
\begin{eqnarray*}
xT(y) = xT_1(y) = T_1(xy) = T(xy) - T_0(xy) = T(xy) 
= T_1(x)y = T(x)y.
\end{eqnarray*}

Next suppose that $A_1$ is a subalgebra.
If $T$ is a multiplier, then for any $x, y \in A$ we have
\begin{equation}
T_1(xy) + T_0(xy) = T(xy) = xT(y) = x_1T_1(y),
\end{equation}
where $x = x_1+x_0$ with $x_1 \in A_1$ and $x_0 \in A_0$.
Here, $x_1T_1(y) \in A_1$ because $A_1$ is a subalgebra, and 
thus, we have $T_0(xy) = x_1T_1(y) - T_1(xy) \in A_0 \cap A_1 = \{0\}$.  Since 
$x, y$ are arbitrary, we get $T_0(A^2) = \{0\}$.  Moreover, $T_1$ is a multiplier because
$ T_1(xy) = x_1T_1(y) = xT_1(y)$ by (9)
and similarly
$T_1(xy) = T_1(x)y$.
The converse is already proved above.

(iii)  Let $S \in (A_1)^{A_1}$ and $x = x_1+x_0, y = y_1 + y_0 \in A$ with $x_1, y_1 \in A_1$ and $x_0, y_0 \in A_0$.  Then, $\pi(x) = \mu(x_1)  = x_1$, $\pi(y) = \mu(y_1) = y_1$ and 
\begin{equation*}
\Phi(S)(x) = \mu(S(\pi(x))) = \mu(S(x_1)) = S(x_1). 
\end{equation*}
Thus, if $S \in M'(A_1)$, we have
\begin{align*}
x\Phi(S)(y) = xS(y_1) = x_1S(y_1) = S(x_1)y_1 = \Phi(S)(x)y_1 = \Phi(S)(x)y. 
\end{align*}
Hence, $\Phi(S) \in M'_1(A)$.  Moreover, if $S \in M(A_1)$, then because $\pi$ is a homomorphism, we have
$$
\Phi(S)(xy) = S(\pi(xy)) = S(x_1y_1) = x_1S(y_1) = x\Phi(S)(y),
$$
and hence $\Phi(S) \in M_1(A)$.
Conversely, let $T \in M'_1(A)$, then because $T$ is a linear mapping 
satisfying $T(A_0) = \{0\}$, 
there is a linear mapping $S \in L(A_1)$ on $A_1$ such that $\Phi(S)  = T$,
that is, $S(x_1) = T(x) = T(x_1)$.
We have 
$$ x_1S(y_1) = x_1T(y_1) = T(x_1)y_1 = S(x_1)y_1, $$
and hence $S \in M'(A_1)$.  Therefore, $\Phi$ is a linear isomorphism from $M'(A_1)$ to $M'_1(A)$.
Similarly, $\Phi$ gives a linear isomorphism from $M(A_1)$ to $M_1(A)$.
Moreover, for $T, U \in M'(A_1)$, we have 
$$\Phi(TU) = \mu\circ T\circ U\circ \pi = \mu\circ T\circ\pi\circ \mu \circ U\circ \pi = \Phi(T)\Phi(U).$$
Thus, $\Phi$ gives an isomorphism of algebras from $M(A_1)$ to $M_1(A)$ and a Jordan isomorphism from $M'(A_1)$ to $M'_1(A)$.
\end{proof}

Theorem 3.1 implies
$$
M'(A) = \ M'_1(A) \oplus M'_0(A),\ M_1(A) \oplus M_0(A) \subseteq M(A),
$$
where $M'_0(A)$ and $M_0(A)$ are given as (8).
Moreover, if $A_1$ is a subalgebra, we have
\begin{equation}
M'(A) \cong M'(A_1) \oplus (A_0)^A,\ \ 
M(A) \cong M(A_1) \oplus \{T \in (A_0)^A\,|\,T(A^2) = \{0\}\}.
\end{equation}

\begin{corollary}
Any weak multiplier $T$ is written as 
\begin{equation}T = T_1 + R\end{equation}
with $T_1 \in LM_1'(A)$ and $R \in (A_0)^{A},$
and it is a multiplier if and only if
\begin{equation}
R(x_1y_1) = x_1T_1(y_1) - T_1(x_1y_1)
\end{equation}
for any $x_1, y_1 \in A_1$.
\end{corollary}
\begin{proof}
As stated above $T$ is written as (11). 
Let $x = x_1+x_0, y = y_1 + y_0 \in A$ with $x_1, y_1 \in A_1$ and $x_0, y_0 \in A_0$ be arbitrary, then we have
\begin{equation}xT(y) = x_1(T_1(y) + R(y)) = x_1T_1(y) = x_1T_1(y_1)\end{equation}
because $R(A) \subseteq A_0$ and $T_1(A_0) = \{0\}$.
The last term in (13) is also equal to $T_1(x_1)y_1 = T(x)y$. 
Hence, $T$ is a multiplier, if and only if it is equal to $T(xy) = T(x_1y_1) = T_1(x_1y_1) + R(x_1y_1)$, if and only if (12) holds.
\end{proof}

\section{Linear multipliers and matrix equation}
In this section, $A$ is a finite dimensional algebra over $K$.
We drive a matrix equation for a linear mapping on $A$ to be a (weak) multiplier.
Suppose that $A$ is $n$-dimensional with basis $E = \{e_1, e_2, \dots, e_n\}$. 

\begin{lemma}  
A linear mapping $T: A \rightarrow A$ is a weak multiplier if and only if
\begin{equation} e_iT(e_j) = T(e_i)e_j, \end{equation}
and it is a multiplier if and only if 
\begin{equation} T(e_ie_j) = e_iT(e_j) = T(e_i)e_j, \end{equation}
for all $e_i, e_j \in E$.
\end{lemma}
\begin{proof}
The necessity of the conditions (14) and (15) is obvious.
Let $x = x_1e_1 + x_2e_2+ \cdots + x_ne_n, y = y_1e_1+y_2e_2+\cdots +y_ne_n \in A$ with $x_1,x_2,\dots,x_n, y_1,y_2,\dots,y_n \in K$.
If $T$ satisfies (14), then we have
\begin{eqnarray*}
xT(y) &=& (\sum_i x_ie_i)T(\sum_j y_je_j) = (\sum_i x_ie_i)(\sum_j y_jT(e_j)) =\sum_{i,j} x_iy_je_iT(e_j)\\
&=& \sum_{i,j}x_iy_jT(e_i)e_j = (\sum_ix_iT(e_i))(\sum_jy_je_j) = T(x)y.
\end{eqnarray*}
Hence, $T$ is a weak multiplier.  Moreover, if $T$ satisfies (15), it is a multiplier in a similar way
\end{proof}
Let $\vv{A}$ (we use the bold character) be the multiplication table of $A$ on $E$. $\vv{A}$ is a matrix whose elements are from $A$ defined by 
\begin{equation} \vv{A} = \vv{E}^t\vv{E}, \end{equation}
where $\vv{E} = (e_1, e_2, \dots, e_n)$ (we again use the bold face $\vv{E}$) is the row vector
consisting the basis elements.
For a linear mapping $T$ on $A$ and a matrix $\vv{B}$ over $A$, $T(\vv{B})$ denotes the matrix obtained by applying $T$ component-wise, that is, the  $(i,j)$-element of $T(\vv{B})$ is $T(b_{ij})$ for the $(i,j)$-element $b_{ij}$ of $\vv{B}$.\footnote
{This is called a broadcasting (cf. \cite{TKTSN}).}
We use the same character $T$ for the representation matrix of $T$ on $E$,
that is,
\begin{equation}
T(\vv{E}) = \vv{E}T.\end{equation}.\vspace{-3mm}

\begin{theorem}
A linear mapping $T$ is a weak multiplier of $A$ if and only if
\begin{equation}\vv{A}T = T^t\vv{A},\end{equation}
and $T$ is a multiplier if and only if
\begin{equation}T(\vv{A}) = \vv{A}T = T^t\vv{A}.\end{equation}
\end{theorem}
\begin{proof}
By (16) and (17) we have
\begin{equation} (e_1, e_2, \dots, e_n)^t(T(e_1), T(e_2), \dots, T(e_n)) = \vv{E}^tT(\vv{E})
= \vv{E}^t\vv{E}T = \vv{A}T\end{equation}
and
\begin{equation}(T(e_1), T(e_2), \dots, T(e_2))^t(e_1, e_2, \dots, e_n) = T(\vv{E})^t\vv{E} = T^t\vv{E}^t\vv{E} = T^t\vv{A}.\end{equation}
By Lemma 4.1, $T$ is a weak multiplier, if and only if (20) and (21) are equal, if and only if (18) holds.
Moreover, $T$ is multiplier if and only if, the leftmost sides of (20) and (21) are equal to
$(T(e_ie_j))_{i,j=1,2,\dots,n} = T(\vv{A})$, if and only if  (19) holds.
\end{proof}

The multiplication table of the opposite algebra $A^{\rm op}$ of $A$ is $\vv{A}^{\rm t}$.  So, the algebras with  multiplication tables transposed to each other have the same (weak) multipliers.

\section{Associative algebras}
In this section $A$ is an associative algebra over $K$.

\begin{proposition}
If $A_0 = \{0\}$, then we have
$$M(A) = M'(A) = LM(A) = LM'(A).$$ 
\end{proposition}
\begin{proof}
Let $T \in M'(A)$,  then we have
$$T(xy)z = xyT(z) = xT(y)z\ \ {\rm and}\ \ zT(xy) = T(z)xy = zT(x)y$$ 
for any $x, y, z \in A$.  
It follows that 
$$T(xy) - xT(y) \in {\rm Ann}_l(A) \cap {\rm Ann}_r(A) = A_0 =\{0\}.$$
Hence, $T(xy) = xT(y)$ and we see $T \in M(A)$.  Moreover, $T \in LM(A)$ by Proposition 2.3.
\end{proof}

Let $a \in A$.  If $xay = axy$ (resp. $xay = xya$) for any $x, y \in A$, $a$ is called a \textit{left} (resp. \textit{right})
\textit{central element}, and $a$ is called a \textit{central element} if $ax = xa$ for any $x \in A$.
Let $Z_l(A)$, (resp. $Z_r(A)$, $Z(A)$) denotes the set of all left central (resp. right central, central) elements.

\begin{lemma}
$Z_l(A)$ (resp. $Z_r(A)$, $Z(A)$) are subalgebra of $A$ containing ${\rm Ann}_l(A)$ (resp. ${\rm Ann}_r(A)$, $A_0$).
\end{lemma}
\begin{proof}
Straightforward.
\end{proof}

For $a\in A$, $l_a$ (resp. $r_a$) denotes the left (resp. right) multiplication by $a$, that is,
$$ l_a(x) = ax,\ \ \ r_a(x) = xa $$
for $x \in A$.  They are linear mappings. 

\begin{lemma}
For $a \in A$ the following statements are equivalent.

(i)  $l_a$ (resp. $r_a$) is a multiplier, 

(ii)  $l_a$ (resp. $r_a$) is a weak multiplier,

(iii)  $a$ is left (resp. right) central.
\end{lemma}
\begin{proof}
If $l_a$ is a weak multiplier, then
$$ xay = xl_a(y) = l_a(x)y = axy $$ 
for any $x, y \in A$. Hence, $a$ is left central.  Because $l_a(x)y = axy = l_a(xy)$ and $xl_a(y) = xay$ for any $x, y \in A$, $l_a$ is a multiplier
if $a$ is left central.
The other case is similar, and we see that the three statements are equivalent.
\end{proof}

\begin{lemma}
Suppose that $A$ has a right (resp. left) identity. Then, a left (resp. right) central element is central, and
${\rm Ann}_l(A) = \{0\}$ (resp. ${\rm Ann}_r(A) = \{0\}$).
\end{lemma}
\begin{proof}  Easy.
\end{proof}

Because ${\rm Ann}_l(A) \subseteq Z_l$ and ${\rm Ann}_r(A) \subseteq Z_r$ by Lemma 5.2, we can make the quotient algebras 
$ \bar{Z}_l(A) =  Z_l(A)/ {\rm Ann}_l(A)\ \,\mbox{and}\ \,\bar{Z}_r(A) = Z_r(A)/{\rm Ann}_r(A).  $

\begin{theorem}
Suppose that $A$ has a left (resp. right) identity $e$.  Then,
any multiplier is a left (resp. right) multiplication by a left (resp. right) central element and is a linear multiplier.  The mapping $\phi: Z_l(A) \rightarrow M(A)$ sending $a \in Z_l(A)$ to $l_a$
induces an isomorphism $\bar{\phi}: \overline{Z_l}(A) \rightarrow M(A)$ of algebras.  
In particular, if $A$ is unital, $M(A)$ is isomorphic to $Z(A)$.
\end{theorem}
\begin{proof}
Suppose that $A$ has a left identity $e$.  Let $T \in M'(A)$ and set $a = T(e)$.
Then we have
$$ T(x) = eT(x) = T(e)x = ax$$
for any $x \in A$.  Hence, $T = l_a$, where $a \in Z_l(A)$ and $T$ is a linear multiplier by Lemma 5.3.
Therefore, $M'(A) = LM(A)$ and $\phi$ is surjective.  Moreover, for $a \in Z_l(A)$, $\phi(a) = 0$, if and only if 
$ax = 0$ for any $x \in A$, if and only if $a \in {\rm Ann}_l(A)$.
Thus we have Ker$(\phi) = {\rm Ann}_l(A)$, and $\phi$ induces the desired isomorphism.
Similarly, if $A$ has a right identity, $M(A)$ is isomorphic to $\overline{Z_r}(A)$.
Finally, if $A$ has the identity, then Ann$_l(A) = \mbox{Ann}_r(A) = \{0\}$ and hence $M(A)$ is isomorphic to $Z(A)$.
\end{proof}

\section{3-dimensional associative  algebras}
Over an algebraically closed field $K$ of characteristic not equal to 2, we have, up to isomorphism, 
24 families of associative algebras of dimension 3.   They are
5 unital algebras $U_0, U_1, U_2$, $U_3, U_4$ defined on basis $E = \{e,f,g\}$ by\vspace{-1mm}
$$\small{
\begin{pmatrix}
e&f&g\\f&0&0\\g&0&0
\end{pmatrix}, \ 
\begin{pmatrix}
e&f&g\\f&0&f\\g&-f&e
\end{pmatrix}, \ 
\begin{pmatrix}
e&0&0\\0&f&0\\0&0&g
\end{pmatrix}, \ 
\begin{pmatrix}
e&0&0\\0&f&g\\0&g&0
\end{pmatrix}, \ 
\begin{pmatrix}
e&f&g\\f&g&0\\g&0&0
\end{pmatrix}},
$$
5 curled algebras $C_0, C_1, C_2, C_3, C_4$
defined by\vspace{-1mm}
$$\small{
\begin{pmatrix}
0 & 0 & 0 \\
0 & 0 & 0 \\
0 & 0 & 0
\end{pmatrix}, \ 
\begin{pmatrix}
0 & 0 & 0 \\
0 & 0 & e \\
0 & -e & 0
\end{pmatrix}, \
\begin{pmatrix}
{0} & {0} & {0} \\
{e} & {f} & {0} \\
{0} & {g} & {0}
\end{pmatrix}, \ 
\begin{pmatrix}
{0} & {0} & {0} \\
{0} & {0} & {0} \\
{e} & {f} & {g}
\end{pmatrix}, \
\begin{pmatrix}
{0} & {0} & {e} \\
{0} & {0} & {f} \\
{0} & {0} & {g}
\end{pmatrix},}
$$ 
non-unital 4 straight algebras $S_1, S_2, S_3, S_4$ defined by\vspace{-1mm}
$${\small
\begin{pmatrix}
{f}\ & {g} & {0} \\
{g}\ & {0} & {0} \\
{0}\ & {0} & {0} 
\end{pmatrix},\ 
\begin{pmatrix}
{e}\ & {0} & {0} \\
{0}\ & {g} & {0} \\
{0}\ & {0} & {0} 
\end{pmatrix},\ 
\begin{pmatrix}
{e}\ & {0} & {0} \\
{0}\ & {f} & {0} \\
{0}\ & {0} & {0} 
\end{pmatrix},\ 
\begin{pmatrix}
{e}\ & {f} & {0} \\
{f}\ & {0} & {0} \\
{0}\ & {0} & {0}  
\end{pmatrix},}\vspace{-1mm}
$$
and non-unital 10 families of waved algebras $W_1$, $W_2$, 
$W_4$, $W_5$, $W_6$, $W_7$, $W_8$, $W_9$, $W_{10}$ and
$\big\{W_3(k)\big\}_{k \in H}$
defined by\vspace{-1mm}
$${\small
\begin{pmatrix}
{0} & {0} & {0} \\
{0} & {0} & {0} \\
{0} & {0} & {e}
\end{pmatrix},\ 
\begin{pmatrix}
{0} & {0} & {0} \\
{0} & {0} & {0} \\
{0} & {e} & {0}
\end{pmatrix},\ 
\begin{pmatrix}
{e} & {0} & {0} \\
{0} & {0} & {0} \\
{0} & {0} & {0}
\end{pmatrix},\ 
\begin{pmatrix}
{0} & {0} & {0} \\
{0} & {0} & {0} \\
{0} & {f} & {g}
\end{pmatrix},\ 
\begin{pmatrix}
{0} & {0} & {0} \\
{0} & {0} & {f} \\
{0} & {0} & {g}
\end{pmatrix},}
$$\vspace{-2mm}
$$\small{
\begin{pmatrix}
e & 0 & 0 \\
0 & 0 & 0 \\
0 & f & g
\end{pmatrix},\ 
\begin{pmatrix}
{e} & {0} & {0} \\
{0} & {0} & {f} \\
{0} & {0} & {g}
\end{pmatrix},\ 
\begin{pmatrix}
{0} & {e} & {0} \\
{e} & {f} & {0} \\
{0} & {g} & {0}
\end{pmatrix},\ 
\begin{pmatrix}
{0} & {e} & {0} \\
{e} & {f} & {g} \\
{0} & {0} & {0}
\end{pmatrix}
\ \mbox{and}\ 
\begin{pmatrix}
{0} & {0} & {0} \\
{0} & {e} & {0} \\
{0} & k{e} & {e}
\end{pmatrix},}
$$
respectively, where $H$ is a subset of $K$ such that $K = H\cup -H$ and $H\cap-H = \{0\}$ (see \cite{KSTT} for details).
We determine the (weak) multipliers of them below.

\smallskip
(\small{0})  $A = C_0$ is the zero algebra, so by Proposition 2.6, we have\vspace{-1mm}
$$M'(A)  = {A}^{A}, \ \,M(A)  = \{T \in {A}^{A}\,|\, T(0) = 0 \}$$\vspace{-1mm}
and \vspace{-1mm}
$$LM(A) = LM'(A) = L(A).$$

(i)  The unital algebras $U_0, U_2$, $U_3, U_4$ are commutative, so for such $A$ we have 
$$M(A) = LM(A) = M'(A) = LM'(A) = \{ l_x | x \in A \} \cong A$$
by Theorem 5.5.  For $A = U_1$, we have
$$M(A) = LM(A) = M'(A) = LM'(A) \cong Z(A) = Ke.$$

(ii)  For $A = C_1$,
We have $A_0 = \mbox{Ann}_l(A) = \mbox{Ann}_r(A) =Ke$, 
and a nihil decomposition $A = A_1 \oplus A_0$ with $A_1 = Kf + Kg$.
Let $T_1 \in M'_1(A)$, then by Theorem 3.1, $T_1$ is a linear mapping such that $T_1(Ke) = \{0\}$.  
Let 
\begin{equation}T_1 = \begin{pmatrix}0&0&0\\0&q&r\\0&t&u\end{pmatrix}\end{equation}
with $q,r,t,u \in K$ be the representation matrix of $T_1$ on $E$.
By Theorem 4.2, $T_1$ is a weak multiplier, if and only if
$$\begin{pmatrix}0&0&0\\0&te&ue\\0&-qe&-re\end{pmatrix} = \vv{A}T_1 = T_1^{\rm t}\vv{A}
= \begin{pmatrix}0&0&0\\0&-te&qe\\0&-ue&re\end{pmatrix},$$
if and only if $r  = t = 0$ and $q = u$.
Hence, $M'_1(A) = \{T_q\,\big|\,q \in K\}$, where $T_q  = \small{\begin{pmatrix}0&0&0\\0&q&0\\0&0&q\end{pmatrix}}$.  By Theorem 3.1 we see \vspace{-1mm}
$$M'(A) = \{T_q\,\big|\,q \in K\} \oplus (Ke)^A,$$
and
$$LM'(A) = \left\{\small{\begin{pmatrix}a&b&c\\0&q&0\\0&0&q\end{pmatrix}}\,
\Big|\,a,b,c,q \in K\right\}.$$
By the multiplication table of $A$, we have\ \,$\alpha\beta = (xv-yz)e$\ \,
for $\alpha = xf+yg, \beta = zf+vg \in A_1$ with $x,y,z,v \in K.$ 
By Corollary 3.2, $T \in M'(A)$ is given by $T = T_q + R$ with $R \in (Ke)^A$ and
this $T$ is a multiplier, if and only if
\begin{eqnarray*}
R((xv-yz)e) &=& R(\alpha\beta)\ =\ \alpha T_q(\beta) - T_q(\alpha\beta)\\ &=& \alpha(q\beta) - T_q((xv-yz)e)\ =\ q(xv-yz)e
\end{eqnarray*}
for any $\alpha$ and $\beta$,
if and only if $R(xe) = qxe$ for all $x \in K$.
Let $S_q = \small{\begin{pmatrix}q&0&0\\0&q&0\\0&0&q\end{pmatrix}}$ be the scalar multiplication by $q \in K$.
Then, we see $(T-S_q)(A) \subseteq A_0 = Ke$ 
and 
$$(T-S_q)(xe) = T_q(xe) + R(xe) - S_q(xe) = 0 + qxe -qxe = 0,$$
for any $x \in K$, that is, $(T - S_q)(Ke)= \{0\}$.  Thus, we have
$$M(A) = \{S_q\,\big|\,q \in K\} \oplus \{R \in (Ke)^A\,|\,R(Ke) = \{0\}\},$$
and
$$LM(A) = \left\{\small{\begin{pmatrix}a&b&c\\0&a&0\\0&0&a\end{pmatrix}}\,\Big|\,a,b,c \in K\right\}.$$

(iii)  $A = C_2$: 
Because Ann$_l(A) = Ke$ and Ann$_r(A) = Kg$, we see
$A_0 = \{0\}$.  Hence, any weak multiplier $T$ is a linear multiplier by Proposition 5.1.  By Theorem 4.2, 
\begin{equation}T =\begin{pmatrix}a&b&c\\p&q&r\\s&t&u\end{pmatrix}\end{equation}
is a (weak) multiplier, if and only if
$$\begin{pmatrix}0&0&0\\ae+pf&be+qf&ce+rf\\pg&qg&rg\end{pmatrix} = \vv{A}T = T^{\rm t}\vv{A}
= \begin{pmatrix}pe&pf+sg&0\\qe&qf+tg&0\\re&rf+ug&0\end{pmatrix},$$
if and only if $b = c = p = r = s = t = 0$ and $a = q = u$,
that is, $T$ is the scalar multiplication $S_a$ by $a$.
Consequently,
$$M(A) = M'(A) = LM(A) = LM'(A) = \{S_a\,\big|\,a \in K\} \cong K.$$

(iv)  $C_3$ and $C_4$ are opposed to each other, and have the same (weak) multipliers by Proposition 2.2.  Let $A = C_3$, 
then, $A$ has a left identity $g$, $Z_l(A) = A$ and Ann$_l(A) = Ke+Kf$.  Hence, by Theorem 5.4, 
$$M(A) = M'(A) = LM(A) = LM'(A) = A/(Ke+Kg) = \{S_a\,\big|\,a \in K\}.$$


(v)  $A = S_1$: 
We have $A_0 = \mbox{Ann}_l(A) = \mbox{Ann}_r(A) = Kg$, and
$A = A_1 \oplus A_0$ with $A_1 = Ke + Kf$.
Then, $T_1 \in M'_1(A)$ is a linear mapping with $T(Kg) = \{0\}$.  Let 
\begin{equation}T_1 = \begin{pmatrix}a&b&0\\p&q&0\\0&0&0\end{pmatrix}\end{equation} 
be its representation on $E$. 
$T_1$ is a weak multiplier, if and only if
$$\begin{pmatrix}af+pg&bf+qg&0\\ag&bg&0\\0&0&0\end{pmatrix} = \vv{A}T_1 = T_1^{\rm t}\vv{A}
= \begin{pmatrix}af+pg&ag&0\\bf+qg&bg&0\\0&0&0\end{pmatrix},$$
if and only if $b = 0$ and $a = q$.  Hence,
$$M'(A) = \left\{T_1^{a,p}\,|\,a,p \in K\right\} \oplus (Kg)^A,$$
where $T_1^{a,p} = \small{\begin{pmatrix}a&0&0\\p&a&0\\0&0&0\end{pmatrix}}$.
So, $T \in M'(A)$ is written as
$T = T_1^{a,p} + R$ with $R \in (Kg)^A$, and 
this $T$ is multiplier, if and only if
\begin{eqnarray*}
R(xzf+(xv+yz)g) &=& R(\alpha\beta)\ =\ \alpha T_1^{a,p}(\beta) - T_1^{a,p}(\alpha\beta)\\
&=& \alpha(aze+(pz+av)f) - T_1^{a,p}(xzf+(xv+yz)g)\\
&=& axzf+(pxz+axv+ayz)g - axzf\\
&=& (pxz+a(xv+yz))g
\end{eqnarray*}
for any $\alpha = xe+yf, \beta = ze+vf \in A_1$ with $x,y,z,v \in K$,
if and only if \ \,$R(xf+yg) = (px+ay)g$\ \,for all $x, y \in K$.  Let
$T^{a,p} = \small{\begin{pmatrix}a&0&0\\p&a&0\\0&p&a\end{pmatrix}}$, then $(T-T^{a,p})(A) \subseteq Kg$, and 
\begin{eqnarray*}
(T-T^{a,p})(xf+yg) &=& \left(T_1^{a,p}+R-T^{a,p}\right)(xf+yg)\\
&=& axf + (px+ay)g - (axf+pxg+ayg) = 0.
\end{eqnarray*}
for any $x, y \in K$.  Thus, $(T - T^{a,p})(Kf+Kg) = \{0\}$, and hence
$$M(A) = \{T^{a,p}\,|\,a, p \in K\} \oplus \{R \in (Kg)^A\,|\,R(Kf+Kg) = \{0\}\}.$$
Taking the intersections of $M'(A)$ and $M(A)$ with $L(A)$, we obtain
$$LM'(A) = \left\{\small{\begin{pmatrix}a&0&0\\p&a&0\\s&t&u\end{pmatrix}}\,\Big|\,a,p,s,t,u \in K\right\}$$
and
$$LM(A) = \left\{\small{\begin{pmatrix}a&0&0\\p&a&0\\s&p&a\end{pmatrix}}\,\Big|\,a,p,s \in K\right\}.$$

(vi)  $A = S_2$: 
We have $A_0 = \mbox{Ann}_l(A) = \mbox{Ann}_r(A) = Kg$, and
$A = A_1 \oplus A_0$ with $A_1 = Ke + Kf$.
Let a linear mapping $T_1 \in M'_1(A)$ be represented as (24), then
$T_1$ is a weak multiplier, if and only if
$$\begin{pmatrix}ae&be&0\\pg&qg&0\\0&0&0\end{pmatrix}
= \vv{A}T = T^{\rm t}\vv{A} = \begin{pmatrix}ae&pg&0\\be&qg&0\\0&0&0\end{pmatrix},$$
if and only if $b = p = 0$.  Hence,
$$M'(A) = \{T_1^{a,q}\,\big|\,a,q \in K\}\oplus (Kg)^A,$$
where $T_1^{a,q} = \small{\begin{pmatrix}a&0&0\\0&q&0\\0&0&0\end{pmatrix}}$,
and
$$LM'(A) = \left\{\small{\begin{pmatrix}a&0&0\\0&q&0\\s&t&u\end{pmatrix}}\,\Big|\,a,q,s,t,u \in K\right\}.$$
By Corollary 3.2, a weak multiplier $T$ written as
$T = T_1^{a,q} + R$ for $a,q \in K$ and $R \in (Kg)^A$ is multiplier, if and only if 
\begin{eqnarray*}
R(xze+yvg) &=& R(\alpha\beta)\ =\ \alpha T_1^{a,q}(\beta) -T_1^{a,q}(xze+yvg)\\
&=& (xe+yf)(aze+qvf) - axze\\
&=& yqvg,
\end{eqnarray*}
for any $\alpha = xe+yf, \beta = ze+vf \in A_1$ with $x,y,z,v \in K$,
if only if $R(xe+yg) = qyg$ for all $x, y \in K$.
Let $T^{a,q} = \small{\begin{pmatrix}a&0&0\\0&q&0\\0&0&q\end{pmatrix}}$, then $T^{a,q} \in M(A)$ and we have $(T-T^{a,q})(xe+yg) = 0$
for any $x, y \in K$ in the same way as (v) above.
Hence, $(T-T^{a,q})(Ke+Kg) = \{0\}$, and we have
$$M(A) = \{T^{a,p}\,|\,a,p \in K\} \oplus \{R \in (Kg)^A\,|\,R(Ke+Kg) = \{0\}\}$$
and
\small{$$LM(A) = \left\{\small{\begin{pmatrix}a&0&0\\0&p&0\\0&t&0\end{pmatrix}}\,\Big|\,a,p,t \in K\right\}.$$

(vii)  $A = S_3$: 
We have $A_0 = Kg$ and $A = A_1 \oplus A_0$ with $A_1 = Ke+Kf$.  Since $A_1$ is a subalgebra of $A$, by Theorem 3.1 we obtain 
the equalities (10) in Section 3.
Because $A_1$ is a commutative unital algebra, 
$$M(A_1) = M'(A_1) = A_1 = \left\{\small{\begin{pmatrix}a&0\\0&b\end{pmatrix}}\,\Big|\,a, b \in K\right\}$$
by Theorem 5.5.  Hence,
$$
M'(A) = A_1 \oplus (Kg)^A\ \,\mbox{and}\ \,
M(A) = A_1 \oplus \{T \in (Kg)^A\,|\,T(Ke+Kf) = 0\}.
$$
Intersecting with $L(A)$ we have
$$ LM'(A) = \left\{\small{\begin{pmatrix}a&0&0\\0&b&0\\s&t&u\end{pmatrix}}\,\Big|\,a,b,s,t,u \in K\right\}\ 
{\rm and}\ \ 
LM(A) = \left\{\small{\begin{pmatrix}a&0&0\\0&b&0\\0&0&u\end{pmatrix}}\,\Big|\,a,b,u \in K\right\}.$$

(viii)  $A = S_4$: 
We have $A = A_1 \oplus A_0$ with $A_0 = Kg$ and $A_1 = Ke+Kf$. 
Because $A_1$ is a commutative unital subalgebra of $A$, similarly to above we have
$$
M'(A) = 
A_1 \oplus (Kg)^A 
= \left\{\small{\begin{pmatrix}a&0\\b&a\end{pmatrix}}\,\Big|\,a, b \in K\right\} \oplus (Kg)^A,
$$
\begin{eqnarray*}
M(A)&=&A_1 \oplus \left\{T \in (Kg)^A\,|\,T(A^2) = 0\right\}\\
&=&\left\{\small{\begin{pmatrix}a&0\\b&a\end{pmatrix}}\,\Big|\,a, b \in K\right\} \oplus \{T \in (Kg)^A\,|\,T(Ke+Kf) = 0\},
\end{eqnarray*}
$$ LM'(A) = \left\{\small{\begin{pmatrix}a&0&0\\b&a&0\\s&t&u\end{pmatrix}}\,\Big|\,a,b,s,t,u \in K\right\}\ {\rm and}\ \ 
LM(A) = \left\{\small{\begin{pmatrix}a&0&0\\b&a&0\\0&0&u\end{pmatrix}}\,\Big|\,a,b,u \in K\right\}.$$

(ix)  $A = W_1:$ 
We have
$A = A_1 \oplus A_0$ with $A_0 = Ke+Kf$ and $A_1 = Kg$.
Let $T_1  \in M'_1(A)$, then $T_1$ is a linear mapping with $T_1(A_0) = \{0\}$.  So $T_1$ is determined by $T_1(g) = ag$ with $a \in K$.
Let denote this $T_1$ by $T^a_1$.  We have
$$M'(A) = \{T^a_1\,|\,a \in K\} \oplus (Ke+Kf)^A.$$
A weak multiplier $T = T^a_1 + R$ with $R \in (Ke+Kf)^A$ is a multiplier,
if and only if
$$ R(xye) = R((xg)(yg)) = xgT^a_1(yg) - T^a_1(xye) =axye$$
for all $x, y \in K$, if and only if $R(xe) = axe$ for any $x\in K$. 
Let $T_a = \small{\begin{pmatrix}a&0&0\\0&0&0\\0&0&a\end{pmatrix}}$.  Then,
$(T-T_a)(Ke) = \{0\}$ and it follows that
$$ M(A) = \{T_a\,\big|\,a\in K\} \oplus \{R \in (Ke+Kf)^A\,\big|\, R(Ke) = \{0\}\}.$$
Also we have
$$LM'(A) = \left\{\small{\begin{pmatrix}a&b&c\\p&q&r\\0&0&u\end{pmatrix}}\,\Big|\,a,b,c,p,q,r,u \in K\right\}$$
and
$$LM(A) = \left\{\small{\begin{pmatrix}a&b&c\\0&q&r\\0&0&a\end{pmatrix}}\,\Big|\,a,b,c,q,r \in K\right\}.$$

(x)  $A = W_2$ \footnote{This is the algebra taken up in \cite{Z}}: 
We have $A = A_1 \oplus A_0$ with $A_0 = Ke$ and $A_1 =  Kf+Kg$.
$T \in M'_1(A)$ is a linear mapping with $T(Ke) = \{0\}$.
Let $T$ be represented as (22), then 
$T$ is a weak multiplier if and only if
$$\begin{pmatrix}0&0&0\\0&0&0\\0&qe&re\end{pmatrix} = \vv{A}T = T^{\rm t}\vv{A}
= \begin{pmatrix}0&0&0\\0&te&0\\0&ue&0\end{pmatrix},$$
if and only if $r = t = 0$, $q = u$.  Hence,
$$M'(A) = \{T_q\,|\,q \in K\}\oplus (Ke)^A,$$
where $T_q = \small{\begin{pmatrix}0&0&0\\0&q&0\\0&0&q\end{pmatrix}}$.  So,
$$LM'(A) = \left\{\small{\begin{pmatrix}a&b&c\\0&q&0\\0&0&q\end{pmatrix}}\,\Big|\,a,b,c,q \in K\right\}.$$
A weak multiplier $T = T_q+R$ with $R \in (Ke)^A$ is a multiplier,
if and only if
\begin{eqnarray*}
R(yze) = R(\alpha\beta) = \alpha T_q(\beta) - T_q(yze)
= \alpha(q\beta) = qyze
\end{eqnarray*}
for any $\alpha = xf+yg, \beta = zf+vg \in A_1$ with $x,y,z,v \in K$, if and only if $R(xe) = qxe$ for all $x \in K$.
Let $S_a$ be the scalar multiplication by $a \in K$.  Then, 
$(T-S_a)(Ke) = \{0\}$, and
hence,
$$M(A) = \{S_a\,|\,a \in K\} \oplus \{R \in (Ke)^A\,|\,R(Ke) = \{0\}\}$$
and
$$LM(A) = \left\{\small{\begin{pmatrix}a&b&c\\0&a&0\\0&0&a\end{pmatrix}}\,\Big|\,a,b,c \in K\right\}.$$

(xi)  $A = W_4$: 
 We have
$A = A_1 \oplus A_0$ with $A_0 = Kf + Kg$ and $A_1 = Ke$.
Because $A_1$ is a subalgbra isomorphic to the base field $K$, for $T_1^a \in M(A_1)$ with $a \in K$ given by $T_1^a(e) = ae$, we see
$$M'(A) = \{T_1^a\,|\,a \in K\} \oplus (fK+gK)^A$$
and
$$ M(A) = \{T_1^a\,|\,a \in K\} \oplus \{R \in (fK+gK)^A\,|\,R(Ke) = 0\}$$
by Theorem 3.1.  Taking the intersection with $L(A)$ we have
$$LM'(A) = \left\{\small{\begin{pmatrix}a&0&0\\p&q&r\\s&t&u\end{pmatrix}}\,\Big|\,a,p,q,r,s,t,u \in K\right\}$$
and
$$LM(A) = \left\{\small{\begin{pmatrix}a&0&0\\0&q&r\\0&t&u\end{pmatrix}}\,\Big|\,a,q,r,t,u \in K\right\}.$$

(xii)  $W_5$ and $W_6$ are opposed.  Let $A = W_5$, then
$A = A_1\oplus A_0$ with $A_0 = Ke$ and $A_1 = Kf+Kg$.  Since $A_1$ is a subalgebra of $A$, we have the equalities (10).  Because $A_1$ has a left identity $g$, we have
\begin{equation*}
M(A_1) = LM(A_1) = M'(A_1) = LM'(A_1) \cong (A_1)/Kf \cong Kg
\end{equation*}
by Theorem 5.5.
So, any element in $M(A_1)$ 
is a scalar multiplication $S_1^q$ in $A_1$ by $q \in K$.
By Theorem 3.1 we have
$$M'(A) = \{S_1^q\,|\,q \in K\} \oplus (Ke)^A,$$
$$ M(A) = \{S_1^q\,|\,q \in K\} \oplus \{R \in (Ke)^A\,|\,R(Kf+Kg) = 0\},$$
$$LM'(A) = \left\{\small{\begin{pmatrix}a&b&c\\0&q&0\\0&0&q\end{pmatrix}}\,\Big|\,a,b,c,q \in K\right\}\ \,
{\rm and}\ \,
LM(A) = \left\{\small{\begin{pmatrix}a&0&0\\0&q&0\\0&0&q\end{pmatrix}}\,\Big|\,a,q \in K\right\}.$$

(xiii)  $W_7$ and $W_8$ are opposed.  Let $A = W_7$. 
We see $A_0 = \mbox{Ann}_r(A) = \{0\}$.  Hence, any weak multiplier is a linear multiplier by Proposition 5.1, and
$T$ represented as (23) is a weak multiplier, if and if
$$
\begin{pmatrix}ae&be&ce\\0&0&0\\pf+sg&qf+tg&rf+ug\end{pmatrix}
= \vv{A}T = T^{\rm t}\vv{A} =
\begin{pmatrix}ae&sf&sg\\be&tf&tg\\ce&uf&ug\end{pmatrix},
$$
if and only if
$b = c = p = r = s = t = 0, q = u$.
Therefore,
$$M(A) = LM(A) = M'(A) = LM'(A) = LM'(A) = \left\{\small{\begin{pmatrix}a&0&0\\0&q&0\\0&0&q\end{pmatrix}}\,\Big|\,a,q \in K\right\}.$$

(xiv)  $W_9$ and $W_{10}$ are opposed.  Let $A = W_9$. 
Then, because $A_0 = \mbox{Ann}_l(A) = \{0\}$, any weak multiplier is a linear multiplier and a linear mapping $T$ represented as (23) 
is a weak multiplier if and only if
$$\begin{pmatrix}pe&qe&re\\ae+pf&be+qf&ce+rf\\pg&qg&rg\end{pmatrix} 
= \vv{A}T = T^{\rm t}\vv{A} = \begin{pmatrix}pe&ae+pf+sg&0\\qe&be+qf+tg&0\\re&ce+rf+ug&0\end{pmatrix}
$$
$c = p = r = s = t = 0, a = q = u$.  
Therefore,
$$LM(A) = M(A) = LM'(A) = M'(A) = \left\{\small{\begin{pmatrix}a&b&0\\0&a&0\\0&0&a\end{pmatrix}}\,\Big|\,a,b \in K\right\}.$$

(xv)  $A = W_3(k)$. 
We have $A = A_1 \oplus A_0$ with $A_0 = Ke$ and $A_1 =  Kf+Kg$.
$T \in M'_1(A)$ is a linear mapping with $T(Ke) = \{0\}$.
Let $T$ be represented as (22), then 
$T$ is a weak multiplier if and only if
$$\begin{pmatrix}0&0&0\\0&qe&re\\0&(kq+t)e&(kr+u)e\end{pmatrix}
= \vv{A}T = T^{\rm t}\vv{A} = \begin{pmatrix}0&0&0\\0&(q+kt)e&te\\0&(r+ku)e&ue\end{pmatrix}.$$
If $k = 0$, then the above holds if and only if $r = t$,
and otherwise it holds if and only if $r = t = 0, q =u$.  Thus,
$$M'(A) = \{T_1^{q,r,u}\,\big|\,q,r,u \in K\} \oplus (Ke)^A,\ LM'(A) = \left\{\small{\begin{pmatrix}a&b&c\\0&q&r\\0&r&u\end{pmatrix}}\,\Big|\,a,b,c,q,r,u \in K\right\}$$
if $k = 0$, and
$$M'(A) = \{T_1^{q}\,\big|\,q \in K\} \oplus (Ke)^A,\ \,
LM'(A) = 
\left\{\small{\begin{pmatrix}a&b&c\\0&q&0\\0&0&q\end{pmatrix}}\,\Big|\,a,b,c,q \in K\right\}$$
if $k \neq 0$, where
$$T_1^{q,r,u} = \small{\begin{pmatrix}0&0&0\\0&q&r\\0&r&u\end{pmatrix}}\ \, {\rm and}\ \,
T_1^{q} = \small{\begin{pmatrix}0&0&0\\0&q&0\\0&0&q\end{pmatrix}}.$$

When $k = 0$, $T = T_1^{q,r,u} + R$ with $R \in (Ke)^A$ is multiplier, if and only if
\begin{eqnarray*}
R((xz+yv)e) &=& R(\alpha\beta)\ =\ \alpha T_1^{q,r,u}(\beta) - T_1^{q,r,u}((xz+yv)e)\\
&=& \alpha((qz+rv)f+(rz+uv)g) = (qxz+r(xv+yz)+uyv)e
\end{eqnarray*}
for any $\alpha = xf+yg, \beta = zf+vg \in A_1$ with $x,y,z,v \in K$,
if and only if $q = u, r = 0$ and $R(xe) = qxe$ for all $x \in K$.
While, when $k \neq 0$, $T = T_1^q + R$ with $R \in (Ke)^A$ is a multiplier, if and only if
\begin{eqnarray*}
R((xz + y(kz+v))e)&=&R(\alpha\beta)\ 
=\ \alpha T_1^q(\beta) - T_1^q(xue+y(kz+v)e)\\
&=&\alpha(q\beta)\ =\ q(xz + y(kz+v))e
\end{eqnarray*}
for any $\alpha, \beta$, if and only if $R(xe) = qxe$ for any $x \in K$.
In the both cases, with the scalar multiplication $S_a$ by $a \in K$, we have\ \,$(T-S_a)(Ke) = \{0\}$.
Therefore, for arbitrary $k$ (either $k$ is zero or nonzero) we see
$$M(A) = \{S_a\,\big|\, a \in K\} \oplus \{R \in (Ke)^A\,|\,R(Ke) = \{0\}\}$$
and
$$LM(A) = \left\{\small{\begin{pmatrix}a&b&c\\0&a&0\\0&0&a\end{pmatrix}}\,\Big|\,a,b,c \in K\right\}.$$

\section{3-dimensional zeropotent algebras}
$A$ is a {\it zeropotent algebra} if $x^2 = 0$ for all $x \in A$.
The zeropotent algebra $A$ over $K$ is anti-commutative, that is,
$xy = -yx$ for all $x,y \in A$.  Thus we see
$$ A_0 = \mbox{Ann}_l(A) = \mbox{Ann}_r(A).$$

Let $A$ be a zeropotent algebras of dimension 3 over $K$ with char$(K) \neq 2$.  
Let ${E} = \{e, f, g\}$ be a basis of $A$.  
Because $A$ is anti-commutative, the multiplication table $\vv{A}$ of $A$ on $E$ is given as
$$\vv{A} = \begin{pmatrix}0&\alpha&-\beta\\-\alpha&0&\gamma\\
\beta&-\gamma&0\end{pmatrix},$$
where
$$
\begin{cases}
\gamma = fg = a_{11}e + a_{12}f + a_{13}g\\
\beta = ge = a_{21}e + a_{22}f + a_{23}g\\
\alpha = ef = a_{31}e + a_{32}f + a_{33}g
\end{cases} 
$$
for $a_{ij} \in K$.
We call $A = \small{\begin{pmatrix}a_{11}&a_{12}&a_{13}\\
a_{21}&a_{22}&a_{23}\\
a_{31}&a_{32}&a_{33}.
\end{pmatrix}}$
the {\it structural matrix} of $A$ 
(we use the same symbol $A$ for the algebra and its structural matrix).

\begin{lemma}
If ${\rm rank}(A) \geq 2$, then $A_0 = \{0\}$.
\end{lemma}
\begin{proof}
If ${\rm rank}(A) \geq 2$, at least two of $fg, ge, ef$ are linearly independent.  
Suppose that $\alpha = ef$ and $\beta = ge$ are linearly independent (the other cases are similar). 
If $x = ae +bf +cg$ with $a,b.c \in K$ is in $\mbox{Ann}_l(A)$, then 
$xe = -b\alpha + c\beta$, $xf = a\alpha-c\gamma$ and $xg = -a\beta+b\gamma$ are all zero.
It follows that $a = b = c = 0$.
Hence, we have Ann$_l(A) = \{0\}$ 
and $A_0 = \mbox{Ann}_l(A) = \{0\}$.
\end{proof}

\begin{theorem}  Let $A$ be a zeropotent algebra $A$ of dimension 3 with rank$(A) \geq 2$ over $K$.
Then, any weak multiplier of $A$ is the scalar multiplication $S_a$ for some $a \in K$,
that is, 
$$M(A) = M'(A) = LM(A) = LM'(A) = \{S_a\,|\,a \in K\}.$$
\end{theorem}
\begin{proof}
By Lemma 7.1 and Corollary 2.4, any weak multiplier $T$ is a linear mapping.
Let $T \in L(A)$ be represented as (23).
By Theorem 4.2,
$T$ is a weak multiplier, if and only if $\vv{A}T = T^{\rm t} \vv{A}$, if and only if
\begin{equation}\small{\begin{pmatrix}p\alpha-s\beta&q\alpha-t\beta&r\alpha-u\beta\\
-a\alpha+s\gamma&-b\alpha+t\gamma&-c\alpha+u\gamma\\
a\beta-p\gamma&b\beta-q\gamma&c\beta-r\gamma
\end{pmatrix}  = 
\begin{pmatrix}-p\alpha+s\beta&a\alpha-s\gamma&-a\beta+p\gamma\\
-q\alpha+t\beta&b\alpha-t\gamma&-b\beta+q\gamma\\
-r\alpha+u\beta&c\alpha-u\gamma&-c\beta+r\gamma
\end{pmatrix}}
\end{equation}
holds.
Suppose that $\alpha = ef, \beta = ge$ are linearly independent (the other cases are similar).
Then, because $p\alpha - s\beta = -p\alpha+s\beta$ by comparing the (1,1)-elements of two matrices in (25),  we have  $p = s = 0$.  
Comparing (1,2)-elements and (1,3)-elements, we have 
$q\alpha - t\beta = a\alpha-s\gamma = a\alpha$ and $r\alpha -u\beta = -a\beta+p\gamma = -a\beta$ respectively.  
It follows that $a = q = u$ and $r = t = 0$.
Comparing (2,2)-elements and (3,3)-elements, we see $b = c = 0$.
Consequently, (23) holds if and only if
\ \,$ b = c = p = r = s = t = 0\ \ \mbox{and}\ \ a = q = u$,
that is, $T = S_a$.
\end{proof}

In \cite{KSTT2} we classify the zeropotent algebras of dimension 3 over an algebraically field $K$ of characteristic not equal to 2.
Up to isomorphism, we have 10 families of zeropotent algebras.  They are\vspace{-1mm}
\[
Z_0, Z_1, Z_2, Z_3, \{Z_4(a)\}_{a\in H}, Z_5, Z_6, \{Z_7(a)\}_{a\in H}, Z_8 \mbox{ and } Z_9
\]
defined by the structural matrices 
\vspace{-2mm}
\[\small{
\begin{pmatrix}0&0&0\\0&0&0\\0&0&0\end{pmatrix},
\begin{pmatrix}0&0&0\\0&0&0\\0&0&1\end{pmatrix},
\begin{pmatrix}0&0&1\\0&0&0\\0&0&1\end{pmatrix},
\begin{pmatrix}0&1&0\\-1&0&0\\0&0&0\end{pmatrix},
\begin{pmatrix}0&0&0\\0&1&a\\0&0&1\end{pmatrix},}
\]\vspace{-2mm}
\[\small{
\begin{pmatrix}0&1&0\\0&0&0\\0&0&1\end{pmatrix},
\begin{pmatrix}0&1&1\\0&0&1\\0&0&1\end{pmatrix},
\begin{pmatrix}1&a&0\\0&1&0\\0&0&1\end{pmatrix},
\begin{pmatrix}1&2&2\\0&1&2\\0&0&1\end{pmatrix} \mbox{and }
\begin{pmatrix}1&3&3\\0&1&3\\0&0&1\end{pmatrix},}
\]
respectively.

$Z_0$ is the zero algebra, and
$Z_1$ is isomorphic to the 3-dimensional associative algebra $C_1$, 
and their (weak) multipliers are already determined in Section 6.  The algebras $Z_3$ to $Z_9$ have rank greater or equal to 2,
which are covered by Theorem 7.2.

Thus, only $A = Z_2$ is left to be analyzed.
The multiplication table $\vv{A}$ of $A$ is $\small{\begin{pmatrix}0&g&0\\-g&0&g\\0&-g&0\end{pmatrix}}$.
We see 
$A_0 = \mbox{Ann}_l(A) = \mbox{Ann}_r(A) = K(e+g),$
and we have the nihil decomposition
$A = A_0 \oplus A_1$ with $A_1 = Ke + Kf$.  A weak multiplier $T \in M'_1(A)$
is a linear mapping represented by $\small{\begin{pmatrix}a&b&c\\p&q&r\\0&0&0\end{pmatrix}}$ satisfying\vspace{-2mm}
$$\begin{pmatrix}pg&qg&rg\\-ag&-bg&-cg\\-pg&-qg&-rg\end{pmatrix} = \vv{A}T = T^{\rm t}\vv{A} = \begin{pmatrix}-pg&ag&pg\\-qg&bg&qg\\-rg&cg&rg\end{pmatrix}.$$
Hence, $a = -c= q$ and $b = p = r = 0$.
Let $T_a$ be this linear mapping, then by Theorem 3.1 we have\vspace{-2mm}
$$M'(A) = \{T_a\,|\,a \in K\} \oplus (K(f+g))^A$$
and 
$$LM'(A) = \left\{\small{\begin{pmatrix}a+s&t&-a+u\\0&a&0\\s&t&u\end{pmatrix}}\,\Big|\,a,s,t,u \in K\right\}.$$

By Corollary 3.2, a weak multiplier $T = T_a + R$ with $R \in (K(e+g))^A$ is a multiplier if and only if for any $\zeta = xe+yf$ and $\eta = ze+vf$ with $x, y, z, v \in K$,
\begin{eqnarray*}R((xv-yz)g)&=&R(\zeta\eta)\ =\ \zeta T_a(\eta) - T_a((xv-yz)g)\\ 
&=&\zeta(a\eta) + a(xv-yz)e\ =\ a(xv-yz)(e+g)
\end{eqnarray*}
holds.  It follows that $R(xg) = ax(e+g)$ for all $x \in K$.  Let $S_a$ be the scalar multiplication by $a$, then
$(T-S_a)(A) \subseteq K(e+g)$ and 
$(T - S_a)(Kg) = \{0\}$. 
Hence, we obtain
$$M(A) = \{S_a\,|\,a \in K\} \oplus \{R \in (K(e+g))^A\,|\, R(Kg) = {0}\},$$
and
$$LM(A) = \left\{\small{\begin{pmatrix}a+s&t&0\\0&a&0\\s&t&a\end{pmatrix}}\,\Big|\,a,s,t \in K\right\}
= \left\{\small{\begin{pmatrix}a&b&0\\0&c&0\\a-c&b&c\end{pmatrix}}\,\Big|\,a,b,c \in K\right\}.$$

%
\end{document}